\PassOptionsToPackage{dvipsnames}{xcolor}
\documentclass[a4paper,12pt]{amsart}
\usepackage[utf8]{inputenc}
\usepackage[english]{babel}
\usepackage[autostyle]{csquotes}
\usepackage[style=alphabetic, backend=biber, doi=false, url=false, giveninits=true,   maxnames=5, maxalphanames=3]{biblatex}
\DeclareFieldFormat[thesis]{title}{\emph{#1}}

\usepackage{amsmath}
\usepackage{amsthm}
\usepackage{stmaryrd}
\usepackage{array}
\usepackage{pgfplots}
\usepackage[a4paper, top=2.8cm, bottom=2.8cm, left=2.2cm, right=2.2cm]{geometry}
\usepackage{amssymb}
\usepackage{graphicx}
\usepackage{caption}
\usepackage{sidecap}
\usepackage{mathtools}
\usepackage{mathrsfs}
\mathtoolsset{showonlyrefs}
\usepackage{enumerate}
\usepackage{enumitem}
\usepackage{esint}
\usepackage{tikz,tikz-3dplot}
\usepackage{tikz-cd}
\usepackage{graphicx}
\usepackage{float}
\usepackage{aligned-overset}

\pgfplotsset{compat=1.18}
\usetikzlibrary{arrows.meta,decorations.markings,patterns.meta,calc,shadings,shapes,positioning, fadings}
\tikzset{->-/.style={decoration={markings, 
  mark=at position #1 with {\arrow[line width=2pt]{>}}},postaction={decorate}}}

\usepackage{scalerel}
\usepackage[colorlinks = true, urlcolor = blue, citecolor = PineGreen, linkcolor = blue]{hyperref}
\usepackage{color,soul}
\usepackage{longtable}

\setlist[description]{leftmargin=3mm, style=nextline, itemsep=1mm}

\numberwithin{equation}{section}

	\newtheorem{theorem}{Theorem}[section]
	
	\newtheorem{proposition}[theorem]{Proposition}

	\newtheorem*{theorem*}{Theorem}
	\newtheorem*{proposition*}{Proposition}
	\newtheorem*{lemma*}{Lemma}
	\newtheorem*{corollary*}{Corollary}
	\newtheorem*{conjecture*}{Conjecture}
	\newtheorem*{statement*}{Statement}
	
	\newtheorem*{claim*}{Claim}
\theoremstyle{definition}
	
	\newtheorem{example}[theorem]{Example}
 
	\newtheorem*{definition*}{Definition}

\theoremstyle{remark}
	
\makeatletter
\renewenvironment{proof}[1][\proofname]{\par
  \pushQED{\qed}%
  \normalfont \topsep6\p@\@plus6\p@\relax
  \trivlist
  \item[\hskip\labelsep
        \scshape
    #1\@addpunct{.}]\ignorespaces
}{%
  \popQED\endtrivlist\@endpefalse
}
\makeatother
\colorlet{luigi}{green!50!gray}
\title{Variational models of robust optimal transport}
\author{Luigi De Masi \and Andrea Marchese \and Annalisa Massaccesi}

\begin{document}

\begin{abstract}
This paper introduces two variational formulations for a model of robust optimal transport, that is, the problem of designing optimal transport networks that are resilient to potential damages, balancing construction costs against the benefit of maintaining partial functionality when parts of the network are damaged. We propose a Eulerian formulation, where the network is modeled by a rectifiable measure and recovery plans are represented by 1-dimensional normal currents. This framework allows for changes in the direction of the transportation in response to damages but restricts damages to be characteristic functions of closed sets. We also propose a Lagrangian formulation, where the network is a traffic plan (that is, a measure on the space of Lipschitz curves) and recovery plans are sub-traffic plans. This approach prescribes the network's orientation but allows for a wider class of damages. 
    
    We prove existence of minimizers in both settings. The two models are compared through examples that illustrate their main differences: the Eulerian formulation's necessity for an unoriented network to achieve existence, the Lagrangian formulation's ability to handle general damages and its requirement for a positive distance between the supports of the source and target measures.
\end{abstract}

\maketitle

\section{Introduction and main results}\label{sec:intro}

Branched optimal transport aims at describing networks that span a prescribed distribution of sources and targets and optimize an energy that models the cost of construction of the network. Usually, this energy depends in a subadditive way on the intensity of the flow per unit length. Actually, there are several other features that could be taken into account, especially if we have real models of branching systems as a reference. Among these features, we want to focus here on the \emph{robustness} of a network, that is, its capacity of enduring a major, abrupt damage without a total collapse of the transport.
  
As observed in nature, especially in leaves, a natural system is often willing to pay something to be endowed with a ``plan B''. In mathematical terms this can be translated in a need of balance between optimization of the resources and redundancy. This suggests that even one of the most fundamental characteristics of an optimal network, that is, the absence of cycles, cannot be expected in the case of robust transport. In fact, the redundancy of the network earns a reward, modeled by a \emph{pay-off}, because it allows the good functioning of (at least a part of) the transport in the event of damage.

In this paper, we study robust optimal transport, whose data are as follows.

\begin{itemize}
\item An ambient domain $X$, which is a compact subset of $\mathbb{R}^n$.
\item A boundary datum $\nu$, which is a signed Radon measure in $X$, whose Jordan decomposition in the pair $\nu^-,\nu^+$ gives us the source and the target of the transport, respectively. Notice that it is not assumed at this stage $\nu^-(X)=\nu^+(X)$.
\item A sub-additive, increasing, and lower semi-continuous cost function $\phi \colon [0,+\infty)\to[0,+\infty)$ such that $\phi(0)=0$ and $\phi'(0)=+\infty$. $\phi(\theta)$ represents the construction cost for an edge of length $1$ and capacity $\theta$. In several major works on the subject of branched optimal transport, one sees $\phi=|\cdot|^\alpha$, with $\alpha\in(0,1)$.
\item A sequence $\{f_i\}_{i \in \mathbb{N}}$ of upper semicontinuous functions on $X$ with values in $[0,1]$, which describe the possible damages. The value $f_i(x) \in [0,1]$ is the relative efficiency of any ``route" that passes through $x$: if $f_i(x)=1$, then the network has normal functionality at $x$, while if $f_i(x)=0$, then the network is completely interrupted in that point.
\item A sequence of positive real numbers $a_i$ such that $\sum_{i \in \mathbb{N}} a_i=1$; each $a_i$ is the probability that damage $f_i$ occurs.
\item A \emph{pay-off} function $h \colon \mathbb{N} \times X \to [0,+\infty)$, bounded and continuous with respect to the second variable; the value $h(i,x)$ is the reward per unit mass starting or arriving at $x$, transported through the network in the event of damage $f_i$.
\end{itemize}
We propose two versions of the problem, a Eulerian and a Lagrangian one. These have different properties and we refer to Section \ref{sec:comparison} for a detailed comparison of the two models.\\

Both formulations pose significant conceptual and technical challenges, which, in our opinion, explain why a variational approach to this problem had not been previously addressed in the literature. Indeed, the main difficulty lies not only in proving existence results, but rather in identifying a formulation that is both mathematically sound and meaningful from the practical point of view.

In the Eulerian framework, a natural choice would be to model the network by a rectifiable current and the recovery plans by subcurrents. However, this approach leads to severe compactness issues: in the limit, cancellation phenomena may occur between recovery plans with opposite orientations, so that the weak limit of a sequence of admissible recovery plans need not correspond to a recovery plan for the limit network. To overcome this, we were led to represent the network by an unoriented rectifiable measure and to work with recovery plans that are currents dominated by it, thus breaking the orientation symmetry but ensuring closure under weak convergence.

The Lagrangian formulation, on the other hand, avoids these cancellation problems and allows for a much larger class of damages, but it introduces its own major difficulty: keeping uniform control of the total mass of the measures on paths. Indeed, each recovery plan may potentially add a non-negligible amount of curves, and a naive definition of admissible competitors would not provide any a priori bound on the total transported mass. Establishing such bounds required a careful reformulation of the problem and the introduction of suitable compactness arguments.

These obstacles highlight that the main novelty of the present work does not lie only in the technical aspects of the proofs, but rather in the identification of the correct variational setting in which the problem becomes well-posed.

\subsection{Eulerian formulation}
In this framework the damages are characteristic functions of closed sets, and a competitor for the transport problem is made by a pair $\left(\mu,\{T_i\}_{i\in\mathbb{N}}\right)$; more precisely:
\begin{itemize}
\item Each $f_i$ is the characteristic function of a closed subset of $X$; this means that, in this model, each damage $f_i$ completely shuts-off the network in a relatively open subset $S_i$ of $X$.
\item $\mu$ is a $1$-dimensional rectifiable measure $\mu=\theta \mathcal{H}^1 \llcorner E$, where $E$ is a $1$-rectifiable set, which represents the network to be built. It can be thought as an \emph{unoriented} graph with multiplicity $\theta(x)$, which is the maximum amount of mass that can be transported through $x$.
\item Due to the threat of possible damages, the network $\mu$ is equipped with a sequence of ``backup" or \emph{recovery plans} $\{T_i\}_{i \in \mathbb{N}}$. Each $T_i$ is a 1-dimensional rectifiable current ``contained" in $\mu$, which is bound to avoid the set $S_i$ and moves a portion of $\nu^-$ onto a portion of $\nu^+$. Each $T_i$ can be thought as an \emph{oriented} graph with multiplicity and  represents a possible partial transport of $\nu^-$ into $\nu^+$ through $\mu$ in case of the damage $f_i$.
\end{itemize}
The energy of a competitor $\left(\mu,\{T_i\}_{i\in\mathbb{N}}\right)$ is given by
\begin{equation}\label{intro_en}
\mathbb{E}^\phi_h(\mu,\{T_i\})
\coloneqq
\mathbb{M}^\phi(\mu)-\sum_{i\in\mathbb{N}}a_i\int h(i,x) \, \mathrm{d} |\partial T_i|(x)\,.
\end{equation}
where $|\partial T_i|$ is the total variation measure associated to the signed measure $\partial T_i$ and
\begin{equation}\label{eq:phi-mass}
    \mathbb{M}^\phi(\mu)= \int_E \phi(\theta(x))\, \mathrm{d} \mathcal{H}^1(x)
\end{equation}
represents the cost of network construction, while the negative term is a \emph{pay-off} and represents an average, in probability, of the reward for the mass transported through $\mu$ by recovery plans $T_i$ in the case of the respective damages. The main result of this formulation is the existence of solutions for the energy minimization problem, see Section \ref{sec:notations} for the relevant notations.

\begin{theorem}\label{prop:existence_solutions}
Given $X, S_i,a_i,\nu,\phi,h$ as above, there exists a minimizer of the energy \eqref{intro_en} among $\left(\mu,\{T_i\}_{i \in \mathbb{N}}\right)$, where $\mu$ is a $1$-rectifiable measure and $T_i\in \mathcal{R}_1(X)$ are recovery plans satisfying the following conditions:
\begin{equation}
\|T_i\| \leq \mu,
\quad
\operatorname{supp} T_i \subseteq X \setminus S_i,
\quad
\partial T_i \preceq \nu
\qquad
\forall i \in \mathbb{N}.
\end{equation}
\end{theorem}
We stress here that the inequality $\|T_i\| \leq \mu$ is intended in the sense of positive Radon measures, since $\mu$ is unoriented.
The construction of a minimizer is based on the theory of currents with coefficients in a normed group, originally adapted to the study of transportation networks in \cite{AnMa,AnMa2} and further developed in \cite{MOV,CMMPT20,MMT19,MMST21,LoScWirth25}.

Referring the reader to Section \ref{sec:comparison} for a discussion of the models, here we notice that, in this formulation, the network $\mu$ is un-oriented, that is, it allows for changes of direction in case of damages; if we try to prescribe the orientation by replacing $\mu$ with a current $T$, then the problem may not have solutions, as discussed in Section \ref{sec:orientation}.
Moreover, damages are characteristic functions of closed sets; in Section \ref{sec:prop_damages} we discuss the difficulties in dealing with more general damages in this Eulerian formulation.

\subsection{Lagrangian formulation.}
If one would like to prescribe the orientation of the network and treat more general damages, then a Lagrangian formulation of the problem seems more appropriate.
    The main ingredients of this version are the following.
\begin{itemize}
    \item Each damage $f_i$ is an upper semicontinuous function, without the restriction of being the characteristic function of a closed set.
    \item Any competitor is a couple $(P,\{P_i\}_{i \in \mathbb{N}})$, where both $P$ and each $P_i$ are so called \emph{traffic plans}, namely Radon measures on the space $K(X)$ of $1$-Lipschitz curves from $[0,+\infty)$ to $X$ which are eventually constant, with the additional condition that the average transport times are finite, namely that
    \begin{equation}
        \int_{K(X)} T(\gamma) \, \mathrm{d} P(\gamma)<\infty,
    \end{equation}
    where $T(\gamma)$ is the minimum time after which $\gamma$ is constant.
    Since each curve can be thought as a unit-mass transport from its starting point to its final point, each transport plan represents a ``weighted collection" of oriented transports. As in the Eulerian formulation, $P$ represents the network to be built, while each $P_i$ is a recovery plan in the event of damage $f_i$, which is a submeasure of $P$ and transports a portion of $\nu^-$ onto a portion of $\nu^+$. We stress the fact that, unlike the Eulerian formulation, in this framework the network $P$ is \emph{oriented} and the recovery plans $P_i$ are bound to match this orientation.
\end{itemize}
    The energy of a competitor $(P,\{P_i\}_{i \in \mathbb{N}})$ is defined as
    \begin{equation}\label{eq:energy_lagr}
    \mathbb{E}\big((P,\{P_i\}_{i \in \mathbb{N}})\big)
    \coloneqq
           \mathbb{M}^\phi(P) - \sum_{i=1}^{+\infty} a_i \int h(i,x) \, \mathrm{d} \big((\pi_0)_\# f_iP_i + (\pi_\infty)_\# f_iP_i\big),
    \end{equation}
    where $\mathbb{M}^\phi(P)$ is the $\phi$-mass of $P$ and models the cost of construction of the network, while the negative term is a pay-off and, as in the Eulerian version, it rewards the resilience of the network in case of damages. More precisely, the terms in the energy are defined as follows.
\begin{itemize}
    \item $\mathbb{M}^\phi(P)$ is defined as
        \begin{equation}\label{eq:phi_mass_traffic}
    \mathbb{M}^\phi(P)
    =
    \int_{K(X)} \int_0^{+\infty} \frac{\phi(|\gamma(t)|_P)}{|\gamma(t)|_P} \dot{\gamma}(t) \, \mathrm{d} t \, \mathrm{d} P(\gamma),
\end{equation}
where $|x|_P=P\big( \{ \gamma \in K(X) \colon x \in \operatorname{Im}(\gamma)\}\big)$; see \cite[Chapter 4]{BCMbook09} for a discussion concerning the conditions on $P$ under which $\mathbb{M}^\phi(P)$ can be expressed in a form similar to \eqref{eq:phi-mass}.
    \item $f_i P_i$ is the traffic plan $P_i$ penalized by the damage $f_i$, that is the traffic plan defined as
\begin{equation}\label{eq:f_iP_i}
    f_iP_i(A) \coloneqq
    \int_{A} \inf_{t \in \mathbb{R}^+} f_i(\gamma(t)) \, \mathrm{d} P_i(\gamma)
\qquad
\forall A \subseteq K(X).
\end{equation}
Thus, $f_iP_i$ is the traffic plan $P_i$ where each curve $\gamma$ has a weight, relative to $P_i$, given by the most penalizing value of $f$ on $\gamma$. Notice that the integrand in \eqref{eq:f_iP_i} is measurable because $f_i$ is Borel and the infimum can be taken on $\mathbb{Q}^+$, since $f_i$ is upper semi-continuous, hence the infima on $\mathbb{R}^+$ and on a dense subset coincide.

    \item The functions $\pi_0$ and $\pi_\infty$ associate to each curve $\gamma \in K(X)$ its initial and final points, respectively; the latter is well-defined because curves in $K(X)$ are eventually constant. The symbols $(\pi_0)_\# f_iP_i$ and $(\pi_\infty)_\# f_iP_i$ are the push-forward of the traffic plan $f_iP_i$ through $\pi_0$ and $\pi_\infty$, respectively; thus they represent the mass transported by $P_i$, with the penalization given by damage $f_i$.
    Notice that, while $\pi_0$ is continuous, $\pi_\infty$ is not; nonetheless, the push-forward $(\pi_\infty)_\#$ is well defined since $\pi_\infty$ is measurable being the pointwise limit, as $T \to +\infty$, of the continuous functions $\pi_T$.
\end{itemize}

The main result concerning this Lagrangian formulation is the existence of minimizers for the energy defined above; see Section \ref{sec:notations} for relevant notations.

\begin{theorem}\label{thm:existence_lagr}
Given $X, f_i,a_i,\nu,\phi,h$ as above, assume $\operatorname{dist}(\operatorname{supp} \nu^-,\operatorname{supp}\nu^+)>0$ and $\lim_{t \to +\infty} \phi(t)=+\infty$. Then there exists a minimizer of the energy \eqref{eq:energy_lagr} among $\left(P,\{P_i\}_{i \in \mathbb{N}}\right)$, where $P,P_i$ are traffic plans that satisfy the following conditions:
\begin{equation}
P_i \leq P,
\quad
(\pi_\infty)_\# P_i - (\pi_0)_\# P_i \preceq \nu
\qquad
\forall i \in \mathbb{N}.
\end{equation}
\end{theorem}
As already said, the Lagrangian formulation models a network with prescribed orientation and allows us to deal with a wider class of damages, that is just upper semi-continuous ones. Referring again to \ref{sec:comparison} for a detailed discussion on the models, we only mention here that the hypotheses $\operatorname{dist}(\operatorname{supp} \nu^-,\operatorname{supp}\nu^+)>0$ and $\lim_{t \to +\infty} \phi(t)=+\infty$ in Theorem \ref{thm:existence_lagr} are necessary, as shown by Examples \ref{ex:distance} and \ref{ex:limit}. Instead, these assumptions are not needed for the existence of a minimizer in the Eulerian formulation.

\subsection{Relations with results in the literature}

An overview on the modern theory of branched optimal transport and its variational description was given by Bernot, Caselles and Morel~\cite{BCMbook09}, based on the Lagrangian formulation proposed by Maddalena, Morel and Solimini~\cite{MadSolMor03}, and the Eulerian formulation proposed by Xia~\cite{Xia03}. Many recent works highlight the central role of this problem in a wide range of applications, see for instance \cite{ConGolOtSer18,Bethuel20,DePGolRuf23,BaLeMassOrl23,CosGolKos24}.

During the last decade, a series of works has clarified the well-posedness of branched transport problems, including general formulations (\cite{PaoSteopt06,Pegon17,CDRMSnonlin17,BWgeneral18,MW19,XiaSun25}), regularity (\cite{Xia04interior,BCM08,Xia11bdry,BraSol14,PegSantXia19}), stability (\cite{CDRMcalc19,CDRMmail19,CDRMcpam21}), and uniqueness (\cite{CaldMS23}).

Branched optimal transport with a pay-off was firstly introduced by \cite{XiaXu} as a model for the problem of designing an optimal network which is not expected to carry all the mass from the source to the target, depending on the global benefit with respect to the cost of the transport. From a mathematical point of view, this translates into a loosening of the boundary constraint: one requires that possibly just a part of $\nu$ is transported (in fact they do not assume $\nu^-(X) = \nu^+(X)$) and authors introduce a model with a payoff term that rewards transported mass.

Similar ideas arise in applied models of network resilience, where optimal structures are required to maintain at least partial functionality under damage or fluctuations. Examples include studies of loop-forming optimal networks under perturbations~\cite{KatiforiSzollosiMagnasco2010} and analyses of resilience in transportation infrastructures~\cite{GaninEtAl2017}. See \cite{AD20} for a comprehensive review of the state-of-the-art.

The present paper combines these two perspectives. Building on the variational theory of branched transport, we propose new formulations that include random damage and recovery plans, capturing the trade-off between construction cost and the benefit of maintaining partial functionality. Our results on existence of minimisers extend the classical well-posedness theory to a setting where redundancy and robustness are part of the optimisation process.

\section*{Acknowledgments}

We would like to thank Benedikt Wirth for inspiring discussions. The authors are partially supported by GNAMPA-INdAM. The research of the first and third named authors has been supported by STARS@unipd project “QuASAR – Questions About Structure And Regularity (MASS STARS MUR22 01). The authors have been supported by the PRIN project 2022PJ9EFL "Geometric Measure Theory: Structure of Singular Measures, Regularity Theory and Applications in the Calculus of Variations" CUP:E53D23005860006.
The third named author has carried out part of this research work while visiting Institute for Advanced Study in Princeton, supported by the National Science Foundation under Grant No. DMS-1926686, that we wish to thank.

\section{Notations}\label{sec:notations}
We collect here the relevant notation used in the paper, referring the reader to \cite{Federer} and \cite{BCMbook09} for extensive discussions on the corresponding topics.

\begin{longtable}{@{\extracolsep{\fill}}lp{0.75\textwidth}}
\multicolumn{2}{c}{\textbf{Measures}}\\[4pt]
$\mathbf{1}_{A}$             &  Characteristic function of the set $A$;\\
$\mathcal{H}^s$		 &	$s$-dimensional Hausdorff measure;\\
$\mu \llcorner E$     &	Restriction of the measure $\mu$ to the set $E$;\\
$\mathcal{M}(U), \mathcal{M}_+(U)$	  &	 Space of signed and positive Radon measures on $U$;\\
$\mathcal{M}^k (U)$		  & 	Space of positive $k$-rectifiable Radon measures on $U$, namely measures of the form $\mu=\theta \mathcal{H}^{k} \llcorner E$, where $E \subset U$ is a $k$-rectifiable set and  $\theta \in L^1(\mathcal{H}^k\llcorner E,\mathbb{R}^+)$.
\\
$|\mu|$                 &   Total variation measure of $\mu$;\\
$\mu \preceq \nu$       &   Inclusions of the Jordan decompositions $\mu=\mu^+-\mu^-$, $\nu=\nu^+-\nu^-$ of the two signed measures $\mu,\nu$, namely $\mu \preceq \nu$ if and only if $\mu^+ \leq \nu^+$ and $\mu^- \leq \nu^-$;\\
$\mathbb{M}^\phi(\mu)$          &   $\phi$-mass of $\mu=\theta \mathcal{H}^k \llcorner E$, namely $\mathbb{M}^\phi(\mu)\coloneqq \int_E \phi(\theta(x)) \, \mathrm{d} \mathcal{H}^k(x)$;\\
\multicolumn{2}{c}{\textbf{}}\\[-1pt]
\multicolumn{2}{c}{\textbf{ Currents }}\\[4pt]
$\partial T$            &   Boundary of the current $T$;\\
$\operatorname{supp} T$               &   Support of the current $T$;\\
$\mathbb{M}(T)$                 &   Mass of the current $T$
\\
$\|T\|$				   &	   Mass measure of the current $T$;\\
$\mathcal{D}_k(U)$
						&   Space of $k$-dimensional currents in $U$ ;\\
$\mathcal{R}_k(U)$
						&   Space of $k$-rectifiable currents in $U$;\\
$\llbracket \gamma \rrbracket$             & $1$-rectifiable current associated to the Lipschitz curve $\gamma$;\\
$\llbracket E,\theta, \tau \rrbracket$                                               &    $k$-rectifiable current induced by the $k$-rectifiable set $E$ with multiplicity $\theta$ and unit orientation $\tau$;\\
$\mathbb{M}^\phi(T)$          &   $\phi$-mass of the current $T=\llbracket E, \theta, \tau \rrbracket$, namely $\mathbb{M}^\phi(T)\coloneqq \int_E \phi(\theta(x)) \, \mathrm{d} \mathcal{H}^k(x)$.\\
\multicolumn{2}{c}{\textbf{}}\\[-1pt]
\multicolumn{2}{c}{\textbf{Traffic plans}}\\[4pt]
$K(X)$                &  Space of $1$-Lipschitz curves from $[0,+\infty)$ to the metric space $X$ which are eventually constant;\\
$T(\gamma)$           &  Transport time of the curve $\gamma \in K(X)$, that is the minimum $T\geq 0$ for which $\gamma$ is constant in $[T,+\infty)$;\\
$L(\gamma)$           &  Length of the curve $\gamma \in K(X)$;\\
$|x|_P$               &  Multiplicity of the traffic plan $P$ at $x$: $|x|_P= P\big( \{ \gamma \in K(X) \colon x \in \operatorname{Im}(\gamma)\}\big)$.\\
\end{longtable}

\section{Existence of minimizers for the Eulerian formulation}
In this section we prove Theorem \ref{prop:existence_solutions}.
We first recall the notion of \emph{good decomposition}, that exists - see \cite[Theorem 3.5]{CDRMcalc19} - for a normal current $T$ which is \emph{acyclic}, namely such that there exists no non-trivial current $S$ satisfying
\begin{equation}
\partial S=0,
\qquad
\mathbb{M}(T)
=
\mathbb{M}(T - S) + \mathbb{M}(S).
\end{equation}

\begin{proposition}[good decomposition]\label{prop:good_decomposition}
An acyclic normal $1$-dimensional current $T$ on $X$ has a \emph{good decomposition}, that is there exists a measure $\pi \in \mathcal{M}_+(K(X))$ concentrated on simple curves which satisfies:
\begin{subequations}\label{eq:def_good_decomposition}
\begin{equation}\label{eq:good_dec_a}
T= \int_{K(X)} \llbracket \gamma \rrbracket \, \mathrm{d} \pi(\gamma),
\end{equation}
\begin{equation}\label{eq:good_dec_b}
\mathbb{M}(T) = \int_{K(X)} \mathbb{M}(\llbracket \gamma \rrbracket) \, \mathrm{d} \pi(\gamma)
=
\int_{K(X)} \mathcal{H}^1(\operatorname{Im} \gamma) \, \mathrm{d} \pi(\gamma),
\end{equation}
\begin{equation}\label{eq:good_dec_c}
\mathbb{M}(\partial T)
=
\int_{K(X)} \mathbb{M}(\partial \llbracket \gamma \rrbracket) \, \mathrm{d} \pi(\gamma)
=
2 \pi(K(X)).
\end{equation}
\end{subequations}
\end{proposition}
\noeqref{eq:good_dec_a} \noeqref{eq:good_dec_b} \noeqref{eq:good_dec_c}

\begin{proof}[Proof of Theorem \ref{prop:existence_solutions}]
Let us call $m_0$ the infimum in \eqref{intro_en} and let us consider a minimizing sequence $\{\left(\mu^k,\{T_i^k\}_{i \in \mathbb{N}}\right)\}^{k \in \mathbb{N}}$, namely such that
\begin{equation}\label{eq:energy_minimizing}
\lim_{k \to +\infty} \mathbb{E}_{h}^\phi\left(\mu^k,\{T_i^k\}_{i \in \mathbb{N}}\right)
=
m_0.
\end{equation}
We can assume without loss of generality that every $T_i^k$ is acyclic.
In fact, by \cite[Proposition 3.12]{PaoSteopt06} we can remove all cycles from each $T_i^k$ without changing $\partial T_i^k$ and the inclusion $T_i^k \leq \mu^k$, thus without affecting the energy of the minimizing sequence. The proof is divided in few steps.

\begin{description}

\item[$\bullet$ A priori bounds.]


Indeed, of course the null competitor has zero energy, thus $m_0 < + \infty$. On the other hand, since $h$ is bounded and $\nu$ has finite mass, the condition $\partial T_i \preceq \nu$ implies that the pay-off term is uniformly bounded, independently of the competitor. This in particular proves $m_0>-\infty$ and
\begin{equation}\label{e:bound}
    \sup_{k \in \mathbb{N}}\mathbb{M}^\phi(\mu^k) < + \infty.
\end{equation}
Since the uniform bound on $\mathbb{M}(\partial T_i^k)$ follows immediately from the conditions $\partial T_i^k \preceq \nu$, in order to prove 
\[
\sup_{i,k \in \mathbb{N}} \left(\mathbb{M}(T_i^k) + \mathbb{M}(\partial T_i^k) \right) < + \infty,
\]
it remains to estimate the mass of $T_i^k$.

Since $\|T_i^k\| \leq {\mu^k}$ as rectifiable Radon measures, the uniform bound \eqref{e:bound} on $\mathbb{M}^\phi(\mu^k)$ gives the same uniform bound on $\mathbb{M}^\phi(T_i^k)$.
We want to use this information to establish the analogous inequality for the masses $\mathbb{M}(T_i^k)$.
In order to do so, let us fix $i,k \in \mathbb{N}$ and let us call $T_i^k= T$ for simplicity of notations. Since $T$ is $1$-rectifiable, thus $T =\llbracket E,\theta,\tau \rrbracket$, for some $1$-rectifiable set $E$, some density $\theta$ and orientation $\tau$.
Since $T$ is is acyclic, by Proposition \ref{prop:good_decomposition} it has a good decomposition $\pi \in \mathcal{M}_+(K(X))$.
Hence, applying \cite[(3.8)]{CDRMcalc19} we have
\begin{equation}\label{eq:uniform_bound_density}
\begin{aligned}
    \theta(x)
&=
\pi(\{\gamma \in K(X) \colon x \in \operatorname{Im} (\gamma)\})
\\[4pt]
&\leq
\pi(K(X))
\\[4pt]
\overset{\eqref{eq:good_dec_c}}&{=}
\frac{\mathbb{M}(\partial T)}{2}
\\[4pt]
&\leq
\frac{|\nu|(X)}{2}
\eqqcolon
\beta
\qquad
\text{for } \mathcal{H}^1 \text{-a.e. }x \in E.
\end{aligned}
\end{equation}
The hypotheses on $\phi$ imply \cite[Lemma 5.1 (6)]{MMST21} the existence of $c>0$ such that $t \leq c \phi(t)$ for every $t \in [0,\beta]$, thus
\begin{equation}\label{eq:estimate_mass_rectifiable}
\mathbb{M}(T)
=
\int_E \theta(x) \, \mathrm{d} \mathcal{H}^1(x)
\leq
\int_E c\phi(\theta(x)) \, \mathrm{d} \mathcal{H}^1(x)
=
c \mathbb{M}^\phi(T).
\end{equation}
Hence, the uniform bound on $\mathbb{M}^\phi(T)$ proved above yields the uniform bound on $\mathbb{M}(T)$.


\item[$\bullet$ Pre-compactness of the minimizing sequence.]

For every $k \in \mathbb{N}$ we have $\mu^k = \theta_k \mathcal{H}^1 \llcorner E_k$ for some rectifiable set $E_k$ and summable $\theta_k$. We define $\tilde{\mu}^k = \tilde{\theta}_k \mathcal{H}^1 \llcorner E_k$, where
\begin{equation}
\tilde{\theta}_k \coloneqq \min\{\theta_k, \beta\}
\end{equation}
and $\beta$ is defined in \eqref{eq:uniform_bound_density}. Since $\beta$ is the uniform bound on the density of $T_i^k$, we have $T_i^k \leq \tilde{\mu}^k$ for every $i,k \in \mathbb{N}$.

The uniform bound on $\tilde{\theta}_k$ and the same arguments of the previous point imply that the total masses of $\tilde{\mu}^k$ are uniformly bounded. 
Up to replacing $\mu^k$ with $\tilde{\mu}^k$, from now on we assume $\mu^k = \tilde{\mu}^k$ for every $k \in \mathbb{N}$. The replacement is justified because $\mathbb{M}^\phi(\tilde{\mu}^k)\leq \mathbb{M}^\phi({\mu}^k)$ and $\tilde{\mu}^k$ supports the currents $\{T_i^k\}$.


For every $i\in\mathbb{N}$, the sequence $(T_i^{k})^{k \in \mathbb{N}}$ is  relatively compact in the space of normal currents with respect to the flat norm \cite[Theorem 4.2.17]{Federer}, by the uniform bound on $\mathbb{M}(T_i^k) + \mathbb{M}(\partial T_i^k) $. Moreover, since $\sup_k |\mu^k|(X)<+\infty$, a diagonal argument yields the existence of a subsequence $(k_j)_{j \in \mathbb{N}}$, a family $\{T_i\}_{i \in \mathbb{N}}$ of normal currents and a measure $\sigma$ in $\mathcal{M}(X)$ (not necessarily rectifiable a priori) such that
\begin{equation}
T_i^{k_j} 
\xrightarrow[j \to +\infty]{\mathbb{F}}
T_i
\quad
\forall i \in \mathbb{N},
\qquad
\mu^{k_j}\xrightharpoonup[j \to +\infty]{\ast} \sigma.
\end{equation}
Since $\mathbb{M}^\phi$ is lower semicontinuous with respect to the flat convergence \cite{CDRMSnonlin17}, it holds
\begin{equation}
\mathbb{M}^\phi(T_i) \leq
\liminf_{k \to +\infty} \mathbb{M}^\phi(T_i^k)
\qquad
\forall i \in \mathbb{N}.
\end{equation}
and the right-hand side is finite (and indeed uniformly bounded with respect to $i$) by the uniform bound on $\mathbb{M}^\phi(T_i^k)$, which follows from \eqref{e:bound}.
This, together with $\phi'(0)=+\infty$ and \cite[Proposition 2.8]{CDRMSnonlin17} or [Proposition 2.32]\cite{BWgeneral18}, yields in particular the rectifiability of each $T_i$.


\item[$\bullet$ Construction of the limit competitor.]

The sequence $T_i= \llbracket E_i,\theta_i, \tau_i \rrbracket$ of rectifiable currents is given by the previous point. Since for every $i\in\mathbb{N}$ the set $S_i$ is open and $\operatorname{supp} T_i^k \subseteq X \setminus S_i$ for every $k\in\mathbb{N}$, we have $$\operatorname{supp} T_i \subseteq X \setminus S_i.$$  The problem with the compactness of the previous point is that $\sigma$ is not rectifiable a priori. We construct ``by hands" a $1$-rectifiable measure $\mu$ which supports the $T_i$'s, by defining $\mu \coloneqq \theta \mathcal{H}^1 \llcorner E$,  where
\begin{equation}
\theta(x)
\coloneqq
\sup_{i \in \mathbb{N}} \theta_i(x),
\qquad
E = \bigcup_{i \in \mathbb{N}} E_i.
\end{equation}
We first prove that $\mu$ has finite mass. Firstly we define $\mu_N \coloneqq \xi_N \mathcal{H}^1 \llcorner E$, where 
$$\xi_N(x)
\coloneqq
\max_{i=1,\dots,N} \theta_i(x)
\quad
\forall N \in \mathbb{N}.$$
Then by the monotone convergence theorem we deduce
\begin{equation}
\int_E \theta(x) \, \mathrm{d} \mathcal{H}^1(x)
=
\lim_{N \to + \infty}  \int_E \xi_N(x) \, \mathrm{d} \mathcal{H}^1(x)
\leq
\sigma(X),
\end{equation}
where the inequality is a consequence of
\begin{equation}
\int_E \xi_N(x) \, \mathrm{d} \mathcal{H}^1(x)
=
\sum_{i=1}^N \int_{F_i} \theta_i(x) \, \mathrm{d} \mathcal{H}^1(x)
\overset{\|T_i\| \leq \sigma}{\leq}
\sum_{i=1}^N \sigma(F_i)
\leq
\sigma(E),
\end{equation}
where $$F_i \coloneqq \{x \in E \colon \xi_N=\theta_i\} \setminus \bigcup_{j<i} \{x \in E \colon \xi_N=\theta_j\}.$$ By construction we have $\|T_i\| \leq \mu$; thus, in order to prove that $(\mu, \{T_i\}_i)$ is an admissible competitor, it remains to prove that $\mu$ has finite $\phi$-mass and that $\partial T_i \preceq \nu$.
We are going to obtain the former as a consequence of
\begin{equation}\label{eq:lsc_phimass}
    \mathbb{M}^\phi(\mu)\leq\liminf_{k \to +\infty}\mathbb{M}^\phi(\mu^k).
\end{equation}
First of all, using the monotonicity, lower semi-continuity of $\phi$ and the monotone convergence theorem, it holds
\begin{equation}\label{eq:monotone_convergence_phi_mass}
\mathbb{M}^\phi(\mu)
=
\int_E \phi(\theta) \, \mathrm{d} \mathcal{H}^1
=
\lim_{N \to +\infty} \int_E \phi(\theta_N) \, \mathrm{d} \mathcal{H}^1
=
\lim_{N \to +\infty} \mathbb{M}^\phi(\mu_N).
\end{equation}
If $T_i^k=\llbracket E_k,\theta_i^k,\tau_i^k \rrbracket$ for $i,k \in \mathbb{N}$, for every $N,k \in \mathbb{N}$ we can define the rectifiable currents $W_N^k, W_N$ with coefficient in $\mathbb{R}^N$, see \cite[Definition 1.8]{AnMa}, given by
\begin{equation}
    W_N^k\coloneqq \llbracket  E_k, (\theta_1^k,\pm \theta_2^k\dots,\pm\theta_N^k),\tau_1^k \rrbracket,
    \qquad
    W_N\coloneqq \llbracket E, (\theta_1,\pm \theta_2\dots,\pm\theta_N),\tau_1 \rrbracket,
\end{equation}
where the signs $\pm$ are chosen in order to obtain $\pm \tau_1^k=\tau_i^k$ and $\pm \tau_1=\tau_i$. The group $\mathbb{R}^N$ is endowed with the maximum norm, namely
\begin{equation}
\|(b_1,\dots,b_N)\|_{\infty}
=
\max_{i=1,\dots,N} |b_i|
\qquad
\forall  (b_1,\dots,b_N)\in \mathbb{R}^N
\end{equation}
so that, for any rectifiable current $T=\llbracket E,(\theta_1,\dots,\theta_N),\tau \rrbracket$ with coefficients in $\mathbb{R}^N$, it is well defined the $\phi$-mass
\begin{equation}
\mathbb{M}^\phi(T)
\coloneqq
\int_M \phi\big(\|(\theta_1(x),\dots,\theta_N(x))\|_\infty\big) \, \mathrm{d} \mathcal{H}^1(x).
\end{equation}
Now let us fix $\varepsilon>0$ and $N \in \mathbb{N}$. By definition of $\mu_N$ and the semicontinuity of $\mathbb{M}^\phi$ for currents with coefficients in $\mathbb{R}^N$ proved in \cite[Proposition 9.4]{MMST21}, there exists $K \in \mathbb{N}$ such that
\begin{equation}
\mathbb{M}^\phi(\mu_N)
=
\mathbb{M}^\phi(W_N)
\leq
\mathbb{M}^\phi(W_N^k) +  \varepsilon
\leq
\mathbb{M}^\phi(\mu^k) + \varepsilon
\qquad
\forall k \geq K.
\end{equation}
The above inequality yields \eqref{eq:lsc_phimass}.


Let us now fix $i \in \mathbb{N}$.
In order to show that $\partial T_i \preceq \nu$ we notice that, since $ \partial T_i^k \preceq \nu$ for every $k \in \mathbb{N}$, by Radon-Nikodym theorem there exist $f_i^k \in L^\infty(|\nu|)$ such that
\begin{equation}\label{eq:f_boundaries}
\|f_i^k\|_{L^\infty} \leq 1,
\quad
\partial T_i^k
=
f_i^k |\nu|
\qquad
\forall k \in \mathbb{N}.
\end{equation}
The weak*-compactness in $L^\infty(|\nu|)$ implies that $f_i^k$ converge in this topology, up to a subsequence, to a function $f_i \in L^\infty(|\nu|)$.
The convergence $\partial T_i^k \to \partial T_i$ as currents yields $\partial T_i=f_i |\nu|$.
In order to complete the proof of $\partial T_i \preceq \nu$, we have to show that
\begin{equation}
    (f_i)^+|\nu| \leq \nu^+
    \quad \text{and} \quad
    (f_i)^-|\nu| \leq \nu^-.
\end{equation}
To do so, we observe that by Jordan decomposition there exist two disjoint sets $A^+,A^- \subseteq X$ such that
\begin{equation}\label{eq:decomposition_nu}
\nu^+ = |\nu|\llcorner A^+,
\qquad
\nu^- = |\nu| \llcorner A^-.
\end{equation}
Since $f_i^k \leq 0$ on $X \setminus A^+$, testing the weak-* convergence of $f_i^k$ to $f_i$ with $\mathbf{1}_{B}$ for any Borel set $B \subseteq X \setminus A^+$, it follows $f_i \leq 0$ on $X \setminus A^+$, hence $(f_i)^+|\nu| \leq \nu^+$. Similarly we have $(f_i)^-|\nu| \leq \nu^-$.


\item[$\bullet$ Continuity of the pay-off.]

To show the lower-semicontinuity of $\mathbb{E}_h^\phi$, by \eqref{eq:lsc_phimass} it suffices to show the continuity of the pay-off term, that is
\begin{equation}\label{eq:continuity_pay-off}
\sum_{i \in \mathbb{N}} a_i \int h(S_i,x) \, \mathrm{d} |\partial T_i|(x)
=
\lim_{k \to + \infty}
\sum_{i \in \mathbb{N}} a_i \int h(S_i,x) \, \mathrm{d} |\partial T_i^k|(x).
\end{equation}

Since $\partial T_i^k \preceq \nu$, by \eqref{eq:f_boundaries} and \eqref{eq:decomposition_nu} it follows that

\begin{equation}
|f_i^k| = f_i^k \big( \mathbf{1}_{A^+} - \mathbf{1}_{A^-} \big)
\qquad
\forall i,k \in \mathbb{N}.
\end{equation}
Hence, by weak-* convergence of $f_i^k$ to $f_i$, we have
\begin{equation}
\begin{split}
\lim_{k \to + \infty}
\int h(S_i,x) \, \mathrm{d} |\partial T_i^k|(x)
= &
\lim_{k \to + \infty}
\int h(S_i,x) |f_i^k| \, \mathrm{d} |\nu|
\\
= &
\lim_{k \to + \infty}
\int h(S_i,x) f_i^k \big( \mathbf{1}_{A^+} - \mathbf{1}_{A^-} \big)\, \mathrm{d} |\nu|
\\
= &
\int h(S_i,x) f_i \big( \mathbf{1}_{A^+} - \mathbf{1}_{A^-} \big)\, \mathrm{d} |\nu|
\\
=&
\int h(S_i,x) \, \mathrm{d} |\partial T_i|.
\end{split}
\end{equation}
Therefore \eqref{eq:continuity_pay-off} follows by the dominated convergence Theorem. \qedhere

\end{description}
\end{proof}

\section{Existence of minimizers for the Lagrangian formulation}

In this section, we prove Theorem \ref{thm:existence_lagr}. The proof follows the existence for the Eulerian formulation, but there are some differences due to the framework.

\begin{proof}[proof of Theorem \ref{thm:existence_lagr}]
    Let us consider a minimizing sequence $(P^k,\{P_i^k\}_i)^k$ of competitors. Since each $P_i^k$ satisfies $(\pi_\infty)_\# P_i^k - (\pi_0)_\# P_i^k \preceq \nu$, the pay-off term in \eqref{eq:energy_lagr} is uniformly bounded, hence
    \begin{equation}
        \sup_{k \in \mathbb{N}} \mathbb{M}^\phi(P^k) =C < +\infty.
    \end{equation}
    Any $P^k$ has finite $\phi$-mass, then one can apply \cite[Proposition 4.6]{BCMbook09} to obtain a \emph{loop-free} traffic plan $\Tilde{P}^k$ supported on injective curves and such that
    \begin{equation}
        \mathbb{M}^\phi(\Tilde{P}^k) \leq \mathbb{M}^\phi(P^k).
    \end{equation}
    More precisely, this process substitutes each curve $\gamma$ with another curve $\Tilde{\gamma}$ with the same endpoints, where the loops of $\gamma$ are cut-off. This induces a similar replacement of $P_i^k$ with another traffic plan $\tilde P_i^k$, which charges these new curves. This substitution possibly increases $f_iP_i^k$ since $\inf_{t \in \mathbb{R}^+}f_i(\gamma(t)) \leq \inf_{t \in \mathbb{R}^+}f(\tilde{\gamma}(t))$. Subsequently, it holds
    \begin{equation}
        \int h(i,x) \, \mathrm{d} \big((\pi_0)_\# f_iP_i^k + (\pi_\infty)_\# f_iP_i^k\big)
        \leq
        \int h(i,x) \, \mathrm{d} \big((\pi_0)_\# f_i\tilde P_i^k + (\pi_\infty)_\# f_i \tilde P_i^k\big)
    \end{equation}
    as well. Therefore $(\tilde P^k,\{\tilde P_i^k\}_i)^k$ is a new minimizing sequence for the problem and we can assume, without loss of generality, that coincides with $(P^k,\{P_i^k\}_i)^k$.

    Since for a loop-free traffic plan $P$ it holds
    \begin{equation}
        \mathbb{M}^\phi(P)
        =
        \int \phi(|x|_P) \, \mathrm{d} \mathcal{H}^1(x),
    \end{equation}
    we have
    \begin{equation}\label{eq:bound_phimass_P_i}
        \mathbb{M}^\phi(P_i^k) \leq \mathbb{M}^\phi(P^k)
        \leq C
        \qquad
        \forall i,k \in \mathbb{N}.
    \end{equation}
    After a reparametrization, we assume from now on that curves in $K(X)$ are parametrized by arc-length. Thus, \eqref{eq:bound_phimass_P_i}  and \cite[Lemma 3.39]{BCMbook09} imply that the average transport times $\int_{K(X)}T(\gamma) \, \mathrm{d} P_i^k$ are uniformly bounded and this gives compactness of each sequence $\{P_i^k\}_i$, by \cite[Theorem 3.28]{BCMbook09}.

    We can make that each $P^k$  charges only curves whose endpoints are contained in $\operatorname{supp} \nu^-$ and $\operatorname{supp} \nu^+$ by restricting it to the set
    \begin{equation}
        \left \{ \gamma \in K(X) \colon \gamma(0) \in \operatorname{supp} \nu^-,\,\gamma(T(\gamma)) \in \operatorname{supp} \nu^+ \right\}.
    \end{equation}
    Since $P^k$ is assumed to be loop-free, this restriction has a possibly lower $\phi$-mass and, since $(\pi_\infty)_\# P_i - (\pi_0)_\# P_i \preceq \nu$, it contains the restrictions of $P_i^k$. Thus, we again have a minimizing sequence for the problem.


In order to obtain an upper bound on the masses of $P^k$, we first observe that, if $0<s<t$, then there exists a positive integer $k$ such that $k s < t \leq (k+1)s$; the monotonicity and sub-additivity of $\phi$ then yield
\begin{equation}
    \frac{\phi(t)}{t} \leq \frac{\phi\big((k+1)s\big)}{ks} \leq \frac{(k+1)\phi(s)}{k s} \leq
    \frac{\phi(s)}{2s},
\end{equation}
hence $\frac{\phi(t)}{t} \leq \frac{\phi(s)}{2s}$. This estimate and the assumption $D \coloneqq \operatorname{dist}(\operatorname{supp} \nu^-,\operatorname{supp}\nu^+)>0$ provide
\begin{equation}\label{eq:uniform_bound_mass_P^k}
\begin{aligned}
       \mathbb{M}^\phi(P^k)
   &=
    \int_{K(X)} \int_0^{+\infty} \frac{\phi(|\gamma(t)|_P)}{|\gamma(t)|_P} \dot{\gamma}(t) \, \mathrm{d} t \, \mathrm{d} P(\gamma)
    \\
   &\geq
   \frac{\phi\big(\mathbb{M}(P^k)\big)}{2 \mathbb{M}(P^k)} \int_{K(X)} L(\gamma) \, \mathrm{d} P^k(\gamma)
   \\
   &\geq
   \frac{D}{2} \cdot \phi\big(\mathbb{M}(P^k)\big)
\end{aligned}
\end{equation}
Since $\lim_{t \to +\infty}\phi(t)=+\infty$, the uniform bound on $\mathbb{M}^\phi(P^k)$ implies that the masses $\mathbb{M}(P^k)$ are uniformly bounded.


Again \eqref{eq:uniform_bound_mass_P^k} then yield a uniform bound on
\begin{equation}
    \int L(\gamma) \, \mathrm{d} P(\gamma)
    =
    \int T(\gamma) \, \mathrm{d} P(\gamma),
\end{equation}
where $T(\gamma)$ is the stopping time of $\gamma$ and the equality holds since we assumed that $\gamma$ are parametrized by arc-lenght.
These considerations give pre-compactness of the sequence $\{P^k\}^k$ again by \cite[Theorem 3.28]{BCMbook09}.
Thus, up to extracting a diagonal subsequence (not relabeled), we can assume that
\begin{equation}
    P^k \overset{*}{\rightharpoonup} P,
    \qquad
    P_i^k \overset{*}{\rightharpoonup} P_i \quad
    \forall i \in \mathbb{N}
\end{equation}
and we are going to show that $(P,\{P_i\}_i)$ is a minimizer for our problem.

Since $P^k,P_i^k$ are positive Radon measures such that $P_i^k \leq P^k$, passing to the limit we have $P_i \leq P$ for every $i \in \mathbb{N}$. Moreover, since the average transport times are uniformly bounded, \cite[Proposition 3.27]{BCMbook09} implies
$$
(\pi_\infty)_\# P_i - (\pi_0)_\# P_i \preceq \nu
\qquad
\forall i \in \mathbb{N}.
$$
Hence $(P,\{P_i\}_i)$ is a competitor for the minimization problem. Since by \cite[Proposition 3.40]{BCMbook09} it holds
\begin{equation}
    \mathbb{M}^{\phi}(P) \leq \liminf_{k \to +\infty}
    \mathbb{M}^{\phi}(P^k),
\end{equation}
in order to show that $(P,\{P_i\}_i)$ is a minimizer, it remains to prove that
\begin{equation}\label{eq:semicont_pay-off}
    \sum_{i=1}^{+\infty} a_i \int h(i,x) \, \mathrm{d} \big((\pi_0)_\# f_iP_i + (\pi_\infty)_\# f_iP_i\big)
    \geq
    \limsup_{k \to +\infty}
    \sum_{i=1}^{+\infty} a_i \int h(i,x) \, \mathrm{d} \big((\pi_0)_\# f_iP_i^k + (\pi_\infty)_\# f_iP_i^k\big).
\end{equation}
To this aim, let us fix $i \in \mathbb{N}$. We first observe that the map $\pi_0 \colon K(X) \to X$ is continuous. If we call $\tilde{f}_i \colon K(X) \to [0,1]$ the function defined as
\begin{equation}
    \tilde{f}_i (\gamma) = \inf_{x \in \operatorname{Im}(\gamma)} f(x),
\end{equation}
we can write
\begin{equation}\label{eq:semicont_pi0}
    \begin{aligned}
         \int_X h(i,x) \, \mathrm{d} (\pi_0)_\# f_iP_i(x)
         &=
        \int_{K(X)} h\big(i,\pi_0(\gamma)\big) \tilde{f}_i(\gamma) \, \mathrm{d} P_i(\gamma)
        \\
        &\geq
        \limsup_{k \to +\infty}  \int_{K(X)} h\big(i,\pi_0(\gamma)\big) \tilde{f}_i(\gamma) \, \mathrm{d} P_i^k(\gamma)
        \\
        &=
        \limsup_{k\to +\infty}\int_X h(i,x) \, \mathrm{d} (\pi_0)_\# f_iP_i^k(x),
    \end{aligned}
\end{equation}
where the inequality follows by standard measure theory, see for instance \cite[Proposition 1.62]{AmbFusPal}, since $X$ is compact and $\tilde{f}_i$ is upper semi-continuous with respect to the topology induced on $ K(X) $ by the uniform convergence on compact sets, thus the integrand is upper semi-continuous.

Unfortunately, the same argument cannot be applied to $(\pi_\infty)_\# f_iP_i^k$ because $\pi_\infty$ is not continuous, so the map $\gamma \mapsto h\big(i,\pi_\infty(\gamma)\big)$ has in general no continuity properties.
On the other hand, exploiting the uniform boundedness of the average transport times, it is possible to obtain the desired upper semi-continuity of the integrals. Indeed, from the previous steps, there exist $C>0$ such that
\begin{equation}
    \int T(\gamma) \, \mathrm{d} P_i^k(\gamma) \leq C
    \qquad
    \forall k \in \mathbb{N}.
\end{equation}
Let us now fix $\varepsilon>0$, choose $M>0$ such that $\frac{C}{M}<\varepsilon$ and define
\begin{equation}
    E \coloneqq \left\{ \gamma \in K(X) \colon T(\gamma) \leq M\right\}.
\end{equation}
By Chebichev inequality it holds
\begin{equation}
    P_i^k \big( K(X) \setminus E \big)
    \leq \frac{C}{M}< \varepsilon
    \qquad
    \forall k \in \mathbb{N}.
\end{equation}
Using this estimate and the positivity of $h$, we obtain
\begin{equation}\label{eq:semicont_piinfty}
    \begin{aligned}
        \limsup_{k \to +\infty} \int_X h(i,x) \, \mathrm{d} (\pi_\infty)_{\#} f_iP_i^k(x)
        &=
        \limsup_{k \to +\infty} \int_X h(i,\pi_{\infty}(\gamma)) \tilde f_i(\gamma) \, \mathrm{d} P_i^k(\gamma)
        \\
        &\leq
        \varepsilon \cdot \sup h + \limsup_{k \to +\infty}  \int_{E} h(i,\pi_{\infty}(\gamma)) \tilde f_i(\gamma) \, \mathrm{d} P_i^k (\gamma)
        \\
        &\leq
        \varepsilon \cdot \sup h + 
        \limsup_{k \to +\infty} \int_{K(X)} h(i,\gamma(M)) \tilde f_i(\gamma) \, \mathrm{d} P_i^k(\gamma)
        \\
        &\leq
        \varepsilon \cdot \sup h +  \int_{K(X)} h(i,\gamma(M)) \tilde f_i(\gamma) \, \mathrm{d} P_i(\gamma)
        \\
        &\leq
        2\varepsilon \cdot \sup h + \int_{K(X)} h(i,T(\gamma)) \tilde f_i(\gamma) \, \mathrm{d} P_i(\gamma)
        \\
        &=
        2\varepsilon \cdot \sup h + \int_X h(i,x) \, \mathrm{d} (\pi_\infty)_{\#} f_iP_i^k(x),
    \end{aligned}
\end{equation}
where the third inequality follows by the continuity of the map $\gamma \mapsto \gamma(M)$, the upper semi-continuity of $\tilde f_i$ and \cite[Proposition 1.62]{AmbFusPal}, as above. Since $\varepsilon$ is arbitrary, this and \eqref{eq:semicont_pi0} yield
\begin{equation}
   \int h(i,x) \, \mathrm{d} \big((\pi_0)_\# f_iP_i + (\pi_\infty)_\# f_iP_i\big)
    \geq
    \limsup_{k \to +\infty}
\int h(i,x) \, \mathrm{d} \big((\pi_0)_\# f_iP_i^k + (\pi_\infty)_\# f_iP_i^k\big)
\qquad
\forall i \in \mathbb{N}.
\end{equation}
For every $i \in \mathbb{N}$, these terms are uniformly bounded, thus multiplying by $a_i$ and summing over $i \in \mathbb{N}$ proves \eqref{eq:semicont_pay-off}. This implies that $(P,\{P_i\}_i)$ is a minimizer for the problem.
    \end{proof}

\section{Comparison between the models}\label{sec:comparison}

In this section we collect some observations on the differences between the Eulerian and Lagrangian models and provide some examples which motivate the formulations of the problems.


\subsection{Orientation}\label{sec:orientation}

In the Eulerian formulation of the robust transport, the network is modeled by an un-oriented rectifiable measure, which allows transportation on both direction on each ``branch", in case of damages: the condition $\|T_i\| \leq \mu$ allows for changes of orientation when different damages occur.
In the Lagrangian formulation, instead, the network is modeled by a transport plan, which prescribes the orientation of the transport on each curve.

It is interesting to study what would happen if, in the Eulerian version, we fix the same orientation for every $T_i$. More precisely, we could study the minimization problem for an energy similar to \eqref{intro_en}, where the ambient network is modeled by a rectifiable current $T \in \mathcal{R}_1(X)$, instead of a measure $\mu$, and we prescribe that every $T_i$ has the same orientation of $T$, namely $T_i \leq T$ instead of the weaker $\|T_i\| \leq \|T\|$. Surprisingly, this problem may not have a solution, as shown by the following example.

\begin{example}\label{es:non_existence}
Let $X=[-3,3]^2 \subset \mathbb{R}^2$, $\phi(t)= \sqrt{t}$, let $h$ be a constant which we fix later and
\begin{equation}
\nu^+ =
\dfrac{\delta_{(-3,0)} + \delta_{(2,0)}}{2}
,
\qquad
\nu^-
=
\dfrac{\delta_{(3,0)} + \delta_{(-2,0)}}{2}.
\end{equation}
We consider only two damaged sets $S_1,S_2$, each with probability $a_1=a_2=\frac{1}{2}$, represented in Figure \ref{fig:example_non_existence}.
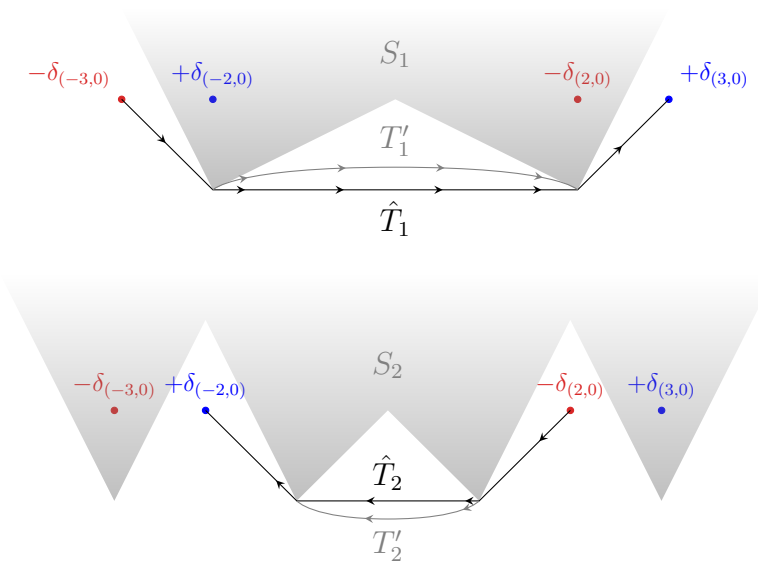
\begin{figure}[ht]
\begin{center}
\begin{tikzpicture}[>=latex,scale=1.2,
decoration={
    markings,
    mark=between positions 0.1 and 0.9 step 1.3cm with {\arrow{stealth}}}
]
\newcommand{\hz}{4}
\newcommand{\er}{0}
\newcommand{\dam}{red!50!white}
\newcommand{\sh}{7}
\draw[fill, color=red!70!gray] (-3,0) circle [radius=1pt] node[above left] {\scalebox{0.8}{$-\delta_{(-3,0)}$}};
\draw[fill, color=blue] (-2,0) circle [radius=1pt] node[above] {\scalebox{0.8}{$+\delta_{(-2,0)}$}};
\draw[fill, color=blue] (3,0) circle [radius=1pt] node[above right] {\scalebox{0.8}{$+\delta_{(3,0)}$}};
\draw[fill, color=red!70!gray] (2,0) circle [radius=1pt] node[above] {\scalebox{0.8}{$-\delta_{(2,0)}$}};

\fill[gray, opacity=0.5, path fading=north]
(-3,1) -- (-2,-1) -- (0,0) -- (2,-1) -- (3,1) -- cycle;
\node [color=gray,opacity=1] at (0,0.5) {$S_1$};
\draw[postaction={decorate}] (-3,0) -- (-2,-1) --node[below] {$\hat{T}_1$} (2,-1) -- (3,0);
\draw [gray, postaction={decorate}] plot [smooth, tension=1.6] coordinates { (-2,-1) (0,-0.75) (2,-1)};
 \node[color=gray, above] at (0,-0.75) {$T_1'$};
\end{tikzpicture}
\\
\bigskip
\begin{tikzpicture}[>=latex,scale=1.2,
decoration={
    markings,
    mark=between positions 0.1 and 0.9 step 1.3cm with {\arrow{stealth}}}
]
\newcommand{\hz}{4}
\newcommand{\er}{0}
\newcommand{\dam}{red!50!white}
\newcommand{\sh}{7}
\draw[fill, color=red!70!gray] (-3,0) circle [radius=1pt] node[above] {\scalebox{0.8}{$-\delta_{(-3,0)}$}};
\draw[fill, color=blue] (-2,0) circle [radius=1pt] node[above] {\scalebox{0.8}{$+\delta_{(-2,0)}$}};
\draw[fill, color=blue] (3,0) circle [radius=1pt] node[above] {\scalebox{0.8}{$+\delta_{(3,0)}$}};
\draw[fill, color=red!70!gray] (2,0) circle [radius=1pt] node[above] {\scalebox{0.8}{$- \delta_{(2,0)}$}};
\fill[gray, opacity=0.5, path fading=north] (-4.25,1.5) -- (-3,-1) -- (-2,1) -- (-1,-1) -- (0,0) -- (1,-1) -- (2,1) -- (3,-1) -- (4.25,1.5) -- cycle;
\node[color=gray,opacity=1] at (0,0.5) {$S_2$};
\draw[postaction={decorate}] (2,0) -- (1,-1) --node[above] {$\hat{T}_2$} (-1,-1) -- (-2,0);
\draw [gray, postaction={decorate}] plot [smooth, tension=1.5] coordinates { (1,-1) (0,-1.2) (-1,-1)};
 \node[color=gray, below] at (0,-1.2) {$T_2'$};
\end{tikzpicture}
\end{center}
\caption{The measures $\nu^-$ (in red) and $\nu^+$ (in blue) and the sets $S_1,S_2$ of the example \ref{es:non_existence}. The pattern $\hat{T}_1$, $\hat{T}_2$ are individually optimal for the pay-off term, but they cannot be both contained in any \emph{oriented} network $T$, since the two segments between $(-1,-1)$ and $(1,-1)$ have opposite orientations.}\label{fig:example_non_existence}
\end{figure}

 Let $\hat{T}_1$ be the shortest curve connecting $(-3,0)$ to $(3,0)$ which avoids $S_1$, and let $\hat{T}_2$ be the shortest curve connecting $(2,0)$ to $(-2,0)$ which avoids the damage $S_2$.
 Since the pay-off term in \eqref{intro_en} equals
 $$
 \frac{h}{2}\big(\mathbb{M}(\partial T_1) + \mathbb{M}(\partial T_2)\big),
 $$
 to minimize the energy for $h$ large enough, it would be convenient to choose $T_1=\hat{T}_2$ and $T_2=\hat{T}_2$. However, no current $T$ can have both $\hat{T}_1$ and $\hat{T}_2$ as subcurrents, since they contain the segment from $(-1,-1)$ to $(1,-1)$ with opposite orientations.
  We can of course chose $T_1'$ and $T_2'$ which do not overlap and are arbitrarily close to $T_1,T_2$, thus showing that the infimum of the energy is not attained.


\end{example}


 \subsection{Properties of damages}\label{sec:prop_damages}

 In the Eulerian formulation the possible damages completely turn-off open sets, namely they are described by characteristic functions of closed sets; the Lagrangian version of the problem allows for much more general damages, namely upper semi-continuous functions.
 
 A natural question is whether this larger class of damages can be treated by the Eulerian formulation.
 If we translate our Lagrangian model, a competitor for the problem would be $(\mu,\{T_i\}_{i \in \mathbb{N}})$, where $\mu=\theta \mathcal{H}^1 \llcorner E$ is a rectifiable measure and the recovery plans $T_i$ are currents such that $\|T_i\| \leq \mu$ and $\partial T_i \preceq \nu^+-\nu^-$ with energy
        \begin{equation}
            \mathbb{E}(\mu,\{T_i\}) \coloneqq \mathbb{M}^\phi(\mu)-\sum_{i\in\mathbb{N}}a_i\int h(i,x) \, \mathrm{d} |\partial \tilde{T}_i|(x)\,.
        \end{equation}
The currents $\tilde{T}_i$, which appear in the pay-off term, are penalizations of the currents $T_i=\llbracket E,\theta_i,\tau_i \rrbracket$ under the action of $f_i$, for instance, any current $\tilde{T}_i=\llbracket E,\tilde{\theta}_i,\tau_i \rrbracket$ satisfying
\begin{equation}
    \tilde{\theta}_i \leq f_i\theta_i,
     \qquad
    \partial \tilde{T}_i \preceq \partial T_i.
\end{equation}
The following example shows that this minimization problem may not admit a solution.

    \begin{example}\label{ex:non_continuous}
        Let $X=[0,1]^2 \subset \mathbb{R}^2$, $\nu^-=\delta_{(1,0)} + \delta_{(0,1)}$, $\nu^+=\delta_{(0,0)} + \delta_{(1,1)}$, and assume that there is only one damage $f$, that occurs with probability $1$, such that 
        \begin{gather}
            f(x,3x)=f(x,3-3x)=1,
            \\
            f(x,1) = \frac{1}{2}
            \quad
            \forall x \in \left(0,\frac{1}{6} \right) \cup \left(\frac{5}{6},1 \right),
            \\
            f\left(x,y \right)=y^\beta
            \quad
            \forall x\in \left( \frac{3}{8},\frac{5}{8}\right) \,\,\forall y \in [0,1],
        \end{gather}
for some $\beta>0$, as illustrated in Figure \ref{fig:non_continuous}.
\begin{figure}[ht]
\begin{center}
\includegraphics[scale=0.2]{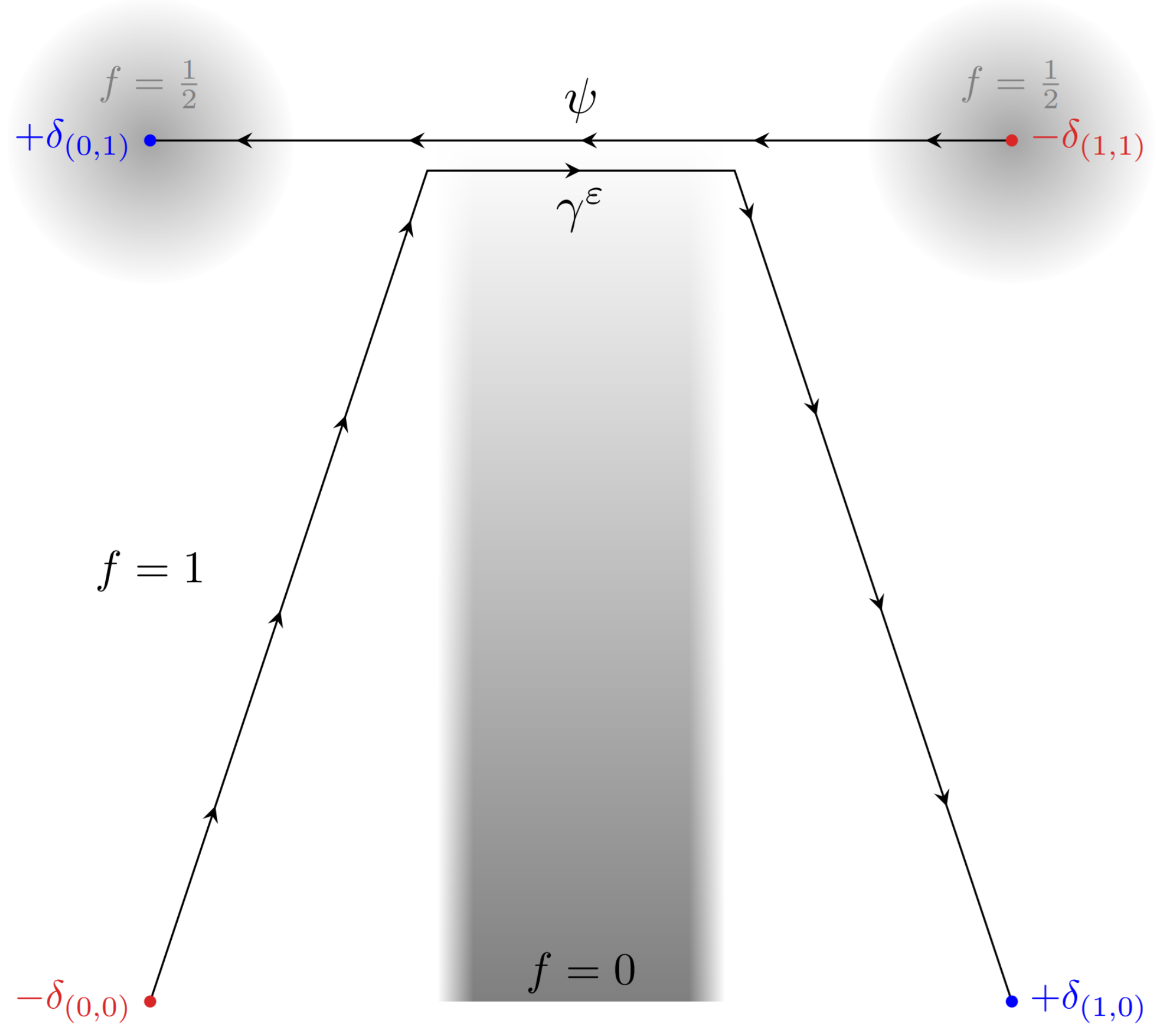}
\end{center}
\caption{The situation described in Example \ref{ex:non_continuous}: the shape of the damage $f$ forces the optimal recovery plan to contain two overlapping curves with different orientations.}\label{fig:non_continuous}
\end{figure}
        If $\phi$, $\beta$ and $h$ are suitably tuned, then an ``almost-optimal" network $\mu_\varepsilon$ is given by the two curves $\gamma^\varepsilon,\psi$ defined as
        \begin{equation}
            \gamma^\varepsilon(x)
            =
            \begin{cases}
                (x,3x)                      & \text{if } 0 \leq x \leq\frac{1}{3}-\varepsilon
                \\[10pt]
                (x,1-3\varepsilon)          & \text{if } \frac{1}{3}-\varepsilon \leq x \leq \frac{2}{3}+\varepsilon
                \\[10pt]
                (x,3-3x)                    & \text{if } \frac{2}{3}+\varepsilon \leq x \leq 1,
            \end{cases}
            \qquad\quad
            \psi(x)=(1-x,1) \quad \forall x \in [0,1],
        \end{equation}
        and the recovery plans $T_1^\varepsilon, \tilde{T}_1^\varepsilon$ are given by
        \begin{equation}
            T_1^\varepsilon
            =
            \llbracket \gamma^\varepsilon\rrbracket + \llbracket \psi \rrbracket,
            \qquad
            \tilde{T}_1^\varepsilon
            =
            (1-3\varepsilon)^\beta \llbracket \gamma^\varepsilon\rrbracket + \frac{1}{2}\llbracket \psi \rrbracket.
        \end{equation}
        Thus, $\tilde{T}_1^\varepsilon$ transports a mass $(1-3\varepsilon)^\beta$ through $\gamma^\varepsilon$ and a mass $\frac{1}{2}$ through $\psi$, with a total amount of approximately $\frac{3}{2}$.
        
        Allowing $\varepsilon \to 0$, the horizontal portion of $\gamma^\varepsilon$ collapses on $\psi$ with opposite orientation, thus $T_1^\varepsilon$ converges to a current $T_1$ which transports $\delta_{(0,1)}$ into $\delta_{(0,0)}$ and $\delta_{(1,0)}$ into $\delta_{(1,1)}$. But this implies that any suitable $\tilde{T}_1$ would transport $\frac{1}{2}\delta_{(0,1)}$ to $\frac{1}{2}\delta_{(0,0)}$ and $\frac{1}{2}\delta_{(1,0)}$ to $\frac{1}{2}\delta_{(1,1)}$, with a total mass of $1$. If $h$ is sufficiently large with respect to $\phi$, then the minization problem does not have a solution.
    \end{example}
In the above example, the problem is the following: as $\varepsilon\to 0$, the curves $\gamma^\varepsilon$ and $\psi$ overlap and the limit transport $T_1$ ``does not respect" the intention of transporting $\delta_{(1,0)}$ to $\delta_{(0,0)}$ and $\frac{1}{2}\delta_{(0,1)}$ to $\frac{1}{2}\delta_{(1,1)}$. Instead, $T_1$ mixes the curves, excluding the overlapping regions.

The necessity of keeping track of the curves is evident, and this suggests that a Lagrangian formulation is the right tool to use.
Indeed, in the Eulerian formulation, a recovery plan $T_i$ does not provide any information on where each part of $\nu^-$ is transported, while in the Lagrangian formulation, each recovery plan tracks the route followed by each portion of $\nu$.

On the other hand, in the Eulerian formulation a recovery plan is a current $T_i$ that can use any part of the network, while in the Lagrangian version of the problem the recovery plans $P_i$ are sub-traffic plans of $P$, thus we allow one to use only curves charged by $P$, not parts of them.

\subsection{Distance between supports of $\nu^-$ and of $\nu^+$ and properties of $\phi$}\label{sec:distance}

 The Eulerian version of the problem does not require any further hypotheses on $\nu$ and $\phi$. In the Lagrangian formulation, instead, we asked
 \begin{equation}
     \operatorname{dist}(\operatorname{supp} \nu^-,\operatorname{supp}\nu^+)>0,
     \qquad
     \lim_{t \to +\infty} \phi(t)=+\infty.
 \end{equation}
 These assumptions are necessary in order to obtain an upper bound on the masses of the traffic plans $P^k$ in the minimizing sequence, as shown by the following examples.

\begin{example}\label{ex:distance}
    Let
    \begin{gather}
        X=[0,1]^2 \subset \mathbb{R}^2,
        \qquad
        \phi(t)=\sqrt{t},
        \qquad
        h \equiv 1, \qquad
        a_i=2^{1-2j} \quad \text{for } j \in \mathbb{N} \text{ and } i=2^{j-1},\dots,2^j-1,
        \\
        x_j=(2^{-3j},0),\qquad y_j=(2^{-3j},2^{-3j}),
        \qquad
        \nu^-=\sum_{j \in \mathbb{N}}2^{-j} \delta_{x_j},
        \qquad
        \nu^+=\sum_{j \in \mathbb{N}}2^{-j} \delta_{y_j}.
    \end{gather}
    Moreover, for every $j \in \mathbb{N}$ and every $i=2^{j-1},\dots,2^j-1$, let $f_i$ assume value $1$ on a curve $\gamma_i$ joining $x_j$ and $y_j$ and $0$ elsewhere and we assume that the curves $\gamma_i$ are mutually disjoint, except for the points $x_j,y_j$. These elements are represented in Figure \ref{fig:distance}.

    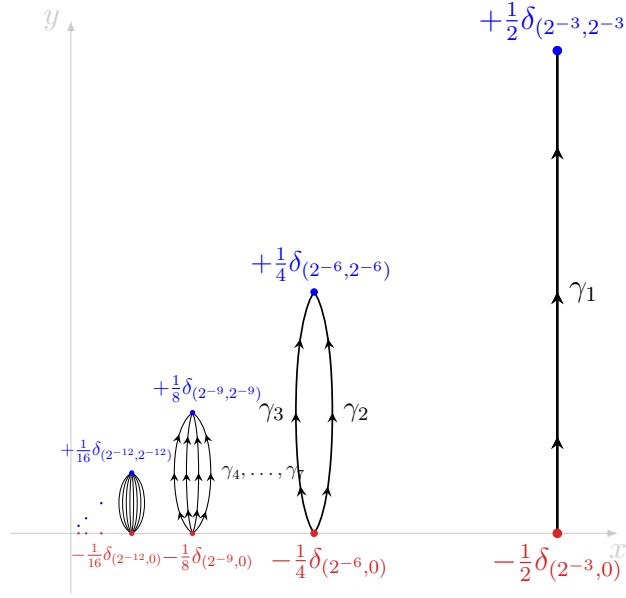
\begin{figure}
        \centering
        \begin{tikzpicture}[>=latex,scale=0.8,
decoration={
    markings,
    mark=between positions 0.2 and 0.9 step 0.3 with {\arrow{stealth}}}
]

\draw[->, gray, opacity=0.4] (-1,0) -- (9,0)node[below]{$x$};
\draw[->,gray, opacity=0.4] (0,-1) -- (0,8.5)node[left]{$y$};

\draw[line width=1pt]
(8,0) --node[right] {$\gamma_1$} (8,8)
[postaction={decorate, decoration={markings,
            mark=between positions 0.2 and 1 step 0.3 with{\arrow{stealth}}
        }}]
;

\draw[postaction={decorate}, line width=0.65pt] plot [smooth, tension=1.5] coordinates { (4,0) (4.3,2) (4,4)};
\node[right] at (4.3,2) {\scalebox{0.9}{$\gamma_2$}};
\draw[postaction={decorate}, line width=0.65pt] plot [smooth, tension=1.5] coordinates { (4,0) (3.7,2) (4,4)};
\node[left] at (3.7,2) {\scalebox{0.9}{$\gamma_3$}};

\draw[postaction={decorate}, line width=0.4pt] plot [smooth, tension=1.5] coordinates { (2,0) (2.30,1) (2,2)};
\draw[postaction={decorate}, line width=0.4pt] plot [smooth, tension=1.5] coordinates { (2,0) (2.1,1) (2,2)};
\draw[postaction={decorate}, line width=0.4pt] plot [smooth, tension=1.5] coordinates { (2,0) (1.9,1) (2,2)};
\draw[postaction={decorate}, line width=0.4pt] plot [smooth, tension=1.5] coordinates { (2,0) (1.7,1) (2,2)};
\node[right] at (2.30,1) {\scalebox{0.65}{$\gamma_4,\dots,\gamma_7$}};

\draw[line width=0.2pt] plot [smooth, tension=1.5] coordinates { (1,0) (1.21,0.5) (1,1)};
\draw[line width=0.2pt] plot [smooth, tension=1.5] coordinates { (1,0) (1.15,0.5) (1,1)};
\draw[line width=0.2pt] plot [smooth, tension=1.5] coordinates { (1,0) (1.09,0.5) (1,1)};
\draw[line width=0.2pt] plot [smooth, tension=1.5] coordinates { (1,0) (1.03,0.5) (1,1)};
\draw[line width=0.2pt] plot [smooth, tension=1.5] coordinates { (1,0) (0.97,0.5) (1,1)};
\draw[line width=0.2pt] plot [smooth, tension=1.5] coordinates { (1,0) (0.91,0.5) (1,1)};
\draw[line width=0.2pt] plot [smooth, tension=1.5] coordinates { (1,0) (0.85,0.5) (1,1)};
\draw[line width=0.2pt] plot [smooth, tension=1.5] coordinates { (1,0) (0.79,0.5) (1,1)};

\draw[fill, color=blue] (8,8) circle [radius=2pt] node[above] {\scalebox{1}{$+\frac{1}{2}\delta_{\left (2^{-3},2^{-3}\right )}$}};
\draw[fill, color=red!70!gray] (8,0) circle [radius=2pt] node[below] {\scalebox{1}{$-\frac{1}{2}\delta_{\left (2^{-3},0\right )}$}};
\draw[fill, color=blue] (4,4) circle [radius=1.5pt] node[above, xshift=1mm] {\scalebox{0.9}{$+\frac{1}{4}\delta_{\left (2^{-6},2^{-6}\right )}$}};
\draw[fill, color=red!70!gray] (4,0) circle [radius=1.5pt] node[below, xshift=2mm] {\scalebox{0.9}{$-\frac{1}{4}\delta_{\left (2^{-6},0\right )}$}};
\draw[fill, color=blue] (2,2) circle [radius=1pt] node[above,xshift=2mm] {\scalebox{0.7}{$+\frac{1}{8}\delta_{\left (2^{-9},2^{-9}\right )}$}};
\draw[fill, color=red!70!gray] (2,0) circle [radius=1pt] node[below,xshift=2mm] {\scalebox{0.7}{$-\frac{1}{8}\delta_{\left (2^{-9},0\right )}$}};
\draw[fill, color=blue] (1,1) circle [radius=1pt] node[above,xshift=-2mm] {\scalebox{0.6}{$+\frac{1}{16}\delta_{\left (2^{-12},2^{-12}\right )}$}};
\draw[fill, color=red!70!gray] (1,0) circle [radius=1pt] node[below,xshift=-2mm] {\scalebox{0.6}{$-\frac{1}{16}\delta_{\left (2^{-12},0\right )}$}};

\draw[fill, color=blue] (0.5,0.5) circle [radius=0.3pt];
\draw[fill, color=blue] (0.25,0.25) circle [radius=0.3pt];
\draw[fill, color=blue] (0.125,0.125) circle [radius=0.3pt];
\draw[fill, color=red!70!gray] (0.5,0) circle [radius=0.3pt];
\draw[fill, color=red!70!gray] (0.25,0) circle [radius=0.3pt];
\draw[fill, color=red!70!gray] (0.125,0) circle [radius=0.3pt];
\end{tikzpicture}
        \caption{A representation of Example \ref{ex:distance}: each damage $f_i$ forces to choose a recovery plan supported on $\gamma_i$, while the pay-off forces any minimizing traffic plan $P$ to use all of them, hence obtaining infinite mass.}
        \label{fig:distance}
    \end{figure}

    Since each damage $f_i$ allows the transportation only through the curve $\gamma_i$, the energy of any competitor $(P,\{P_i\}_i)$ for the minimization problem can be reduced if we restrict the traffic plan $P$ to these curves and if we maximize the mass transported by each $P_i$; hence we can assume
    \begin{equation}
        P= \sum_{i \in \mathbb{N}} m_i \delta_{\gamma_i},
        \qquad
        P_i= \min\{m_i,2^{-j}\}\delta_{\gamma_i}
        \quad \forall i \in \{2^{j-1},\dots,2^j-1\},
    \end{equation}
    for a suitable sequence $\{m_i\}_{i \in \mathbb{N}}$. Thus the energy of the competitor is
    \begin{equation}
        \sum_{j \in \mathbb{N}} \sum_{i=2^{j-1}}^{2^j-1}
        \left( 2^{-3j}\sqrt{m_i} - 2^{1-2j} \min\{m_i,2^{-j}\}\right).
    \end{equation}
    Since the minimum of the function $m_i \mapsto 2^{-3j}\sqrt{m_i} - 2^{1-2j} \min\{m_i,2^{-j}\}$ is attained for $m_i=2^{-j}$, a minimizer for the problem must satisfy this condition for every $j \in \mathbb{N}$ and every $i\in \{2^{j-1},\dots,2^j-1\}$, hence having energy
    \begin{equation}
        \sum_{j \in \mathbb{N}} \sum_{i=2^{j-1}}^{2^j-1} \left( 2^{-3j}2^{-\frac{j}{2}} - 2^{1-2j} 2^{-j}\right)
        =
        \sum_{j \in \mathbb{N}} 2^{-2j} \big(2^{-\frac{j}{2}} - 1 \big),
    \end{equation}
    which converges to a negative number. On the other hand, this choice would produce a traffic plan $P$ with infinite mass, because
    \begin{equation}
        P(K(X))=
        \sum_{j \in \mathbb{N}} \sum_{i=2^{j-1}}^{2^j-1} 2^{-j}
        =
        \sum_{j \in \mathbb{N}} 2^{j-1} 2^{-j}
        =
        +\infty.
    \end{equation}
\end{example}

\begin{example}\label{ex:limit}
Let $\nu^-=\delta_{(0,0)}$, $\nu^+=\delta_{(1,0)}$ and let each damage $f_i$ assume value $1$ on the curve $\gamma_i$ joining $(0,0)$ with $(1,0)$ represented in Figure \ref{fig:limit} and $0$ elsewhere. Assume $\phi(t)=1$ for all $t>0$ so that the $\phi$-mass of a loop free traffic plan is the total length of its support.

Given any $h, a_i$, if the circle described by $\gamma_i$ is sufficiently small, then it is convenient for the energy to set $P_i$ as the traffic plan (of mass $1$) concentrated on $\gamma_i$. However, charging all these curves would produce a minimizing sequence of networks $P^k$ with unbounded masses.

        \begin{figure}[t]
        \centering
        \begin{tikzpicture}[>=latex,scale=0.8,
]

\draw[line width=0.5pt]
(0,0) --node[above, pos=0.75] {$\gamma_1$} (8,0)
[postaction={decorate, decoration={markings,
            mark=between positions 0.1 and 1 step 0.1 with{\arrow{stealth}}
        }}]
;

\draw[line width=0.5pt]
({4-0.5},0) arc (180:0:0.5)
[postaction={decorate, decoration={markings,
            mark=at position 0.5 with{\arrow{stealth}}
        }}]
;

\node[above] at (4,0.5) {$\gamma_2$};

\draw[line width=0.5pt]
({2-0.3},0) arc (180:0:0.25)
[postaction={decorate, decoration={markings,
            mark=at position 0.5 with{\arrow{stealth}}
        }}]
;

\node[above] at (2,0.25) {$\gamma_3$};

\draw[line width=0.5pt]
({1-0.125},0) arc (180:0:0.125)
;

\node[above] at (1,0.125) {$\gamma_4$};

\draw[fill, color=blue] (8,0) circle [radius=2pt] node[right] {\scalebox{1}{$+\delta_{\left (1,0\right )}$}};
\draw[fill, color=red!70!gray] (0,0) circle [radius=2pt] node[left] {\scalebox{1}{$-\delta_{\left (0,0\right )}$}};
\end{tikzpicture}
        \caption{Example \ref{ex:limit}: if $\phi$ is bounded and the curves $\gamma_i$ transporting $\nu^-$ to $\nu^+$ are ``almost overlapping", then for a minimizer it would be convenient to charge all of them, obtaining a traffic plan with infinite mass.}
        \label{fig:limit}
    \end{figure}
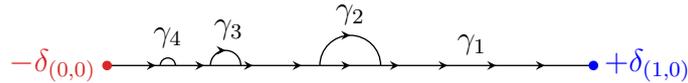
\end{example}

\printbibliography

\end{document}